\documentclass[final,3p,times]{elsarticle}
\usepackage[colorlinks=true]{hyperref}
\usepackage{amsfonts}
\usepackage{mathrsfs}
\usepackage{amsthm,amsmath,amssymb,mathrsfs,amsfonts,graphicx,graphics,latexsym,exscale,cmmib57,dsfont,bbm,amscd,ulem}
\usepackage{color,soul}%%%\hl{ highlighting text } % use hl{ } to highlight the sentence.
\newtheorem{thm}{Theorem}[section]
\newtheorem{lem}[thm]{Lemma}

\newtheorem{prop}[thm]{Proposition}
\newtheorem{cor}[thm]{Corollary}

\theoremstyle{definition}
\newtheorem{de}[thm]{Definition}

\newtheorem*{rem*}{Remark}

\newtheorem{claim4}{Claim}[section]
\newtheorem*{rthm1.6}{Theorem~1.6}
\numberwithin{equation}{section}
\journal{JMAA (May 9, 2017; accepted July 23, 2017)}

\begin{document}

\begin{frontmatter}
\title{Chaotic dynamics of minimal center of attraction for a flow with discrete amenable phase group}

\author[label1,label2]{Zhijing Chen}
\ead{chzhjing@mail2.sysu.edu.cn}
\author[label1]{Xiongping Dai}
\ead{xpdai@nju.edu.cn}

\address[label1]{Department of Mathematics, Nanjing University, Nanjing 210093, People's Republic of China}
\address[label2]{Department of Mathematics, Zhongshan (Sun Yat-Sen) University, Guangzhou 510275, People's Republic of China}
%%%%%%%%%%%%%%%%%%%%%%%%%%%%%%%%%%%%%%%%%%%%%%%%%%%%%%%%%%%%%%%%%%%%%%%%%%%%%%%%%%%%%%
\begin{abstract}
Let $G$ be a discrete infinite amenable group, which acts from the left on a compact
metric space $X$. In this paper, we study the chaotic
dynamics exhibited inside and near a minimal center of attraction of $(G,X)$
relative to any F{\o}lner net in $G$.
\end{abstract}

\end{frontmatter}

%\tableofcontents

\section{Introduction}\label{sec1}%%%
Throughout let $(G,\cdot)$ be a discrete infinite {\it amenable} group. Let $(G,X)$ be a topological dynamical system, or \textit{$G$-system} for short, on a compact
metric space $(X,d)$ (cf.~$\S\ref{sec2.2}$ for the precise definition).
For any point $x\in X$, we shall call  $Gx =\{gx\,|\,g\in G\}$ the {\it orbit} of $x$ under the action of $G$.
We refer to any subset $\Lambda$ of $X$ as a {\it $G$-invariant set}
if $gx\in\Lambda$ for each $x\in\Lambda$ and any $g\in G$.
In dynamical systems, statistical mechanics and ergodic theory,
we often need to do with probability of sojourn of an
orbit $Gx$ in a given region $E$ of $X$.
This drives us to consider density in $(G,X)$.

In order to be more specific let us introduce some basic notation and definitions. First, recall that a net $\mathcal {F}=\{F_n\}_{n\in D}$ of finite subsets of $G$ is called
a (left) {\it F{\o}lner net} in $G$ if
\begin{gather*}
\lim_{n\in D}\frac{|gF_n\vartriangle F_n|}{|F_n|}=0\quad  \forall g\in G,
\end{gather*}
in the sense of net limit, where $|\cdot|$ is the counting measure on $G$.
Clearly every subnet $\{F_{n_k}\}_{k\in E}$ of a F{\o}lner net $\{F_n\}_{n\in D}$
is also a F{\o}lner net in $G$. Since $(G,\cdot)$ is assumed to be amenable here,
it always has a F{\o}lner net (cf.,~e.g., \cite{Argabright,Hindman-06}).

After choosing any F{\o}lner net $\mathcal {F}=\{F_n\}_{n\in D}$ in $G$,
for any subset $A\subseteq G$,
we can define respectively the {\it upper} and {\it lower density of $A$ relative to $\mathcal {F}$} by
\[\overline{d}_{\mathcal{F}}(A) = \sup\left\{\alpha\,|\,\forall m \in  D\, \exists n\geq  m \text{ s.t. }\frac{|A\cap F_n|}{|F_n|}\geq\alpha\right\}\quad
\textrm{and}\quad
\underline{d}_{\mathcal{F}}(A) = \sup\left\{\alpha\,|\,\exists m \in  D\textrm{ s.t. }\forall n\geq  m, \frac{|A\cap F_n|}{|F_n|}\geq\alpha\right\}.\]
If $\overline{d}_{\mathcal{F}}(A)=\underline{d}_{\mathcal{F}}(A)$,
then we call this value the {\it density of $A$ relative to $\mathcal {F}$} and denote it by $d_{\mathcal{F}}(A)$.
Following~\cite{Hindman-06} these notions are well defined.

Now, relative to a F{\o}lner net $\mathcal {F}=\{F_n\}_{n\in D}$ in $G$,
the {\it probability of sojourn of an
orbit} $Gx$ in a given region $E$ of $X$ is described by
\[P_x(E):=d_{\mathcal{F}}(\{g\in G\,|\,gx\in E\})\]
if the $d_{\mathcal{F}}$-density exists.
This motivates us to introduce the following concept for $G$-systems; see~\cite{Hilmy-36} for $\mathbb{R}$-systems and \cite{Dai-2} for $\mathbb{R}_+$-systems.

\begin{de}
Given any $x\in X$ and any F{\o}lner net $\mathcal {F}=\{F_n\}_{n\in D}$ in $G$,
a closed subset $C$ of $X$ is called an {\it $\mathcal {F}$-center of attraction of $x$}
if $P_x(B_\epsilon(C))= 1$ for all $\epsilon > 0$.
If the set $C$ does not admit any proper subset which
is likewise an $\mathscr{F}$-center of attraction of $x$,
then $C$ is called the {\it minimal $\mathcal {F}$-center of attraction of
$x$} and write $\mathcal{C}_{\mathcal{F}}(x)$.
 Here $B_\epsilon(C)$ denotes the $\epsilon$-neighborhood around $C$ in $X$.
\end{de}

Given any $x\in X$, by $\mathscr{U}_x$ we denote the neighborhood system of $x$ in $X$. Inspired by ~\cite{Nemytskii-60,Sigmund-77},
we will prove the following characterization of minimal center of attraction in $\S\ref{sec2.3}$:

\begin{lem}\label{lem1.2}
Given any $x\in X$ and any F{\o}lner net $\mathcal {F}=\{F_n\}_{n\in D}$ in $G$, there holds
\[{\mathcal{C}}_{\mathcal{F}}(x)=\left\{y\in X\,|\,\overline{d}_{\mathcal{F}}(\{g\in G\,|\,gx\in U\})>0\ \forall U\in\mathscr{U}_y\right\}.\]
Consequently, there always exists a minimal center of attraction of a point relative to any F{\o}lner net.
\end{lem}

It is well known that the the minimal center of attraction admits abundant dynamics for
$\mathbb{R}_+$-systems and $\mathbb{Z}_+$-systems;
see, e.g.,~\cite{Nemytskii-60,Sigmund-77, HZ-12,Dai-2}.
In this paper, we will discuss the chaotic behavior of ${\mathcal{C}}_{\mathcal{F}}(x)$ for $G$-systems.

For our convenience,
we first introduce the following two notions (see~\cite[Definitions 1.5 and 1.6]{Dai-2} for $\mathbb{R}_+$-systems):
\begin{itemize}
\item A $G$-invariant subset $\Lambda$ of $X$ is called Karl Sigmund generic ({\it S-generic} for short)
 if there exists some point $x\in\Lambda$ with $\Lambda={\mathcal{C}}_{\mathcal{F}}(x)$.

\item Given any $x,y\in X$,
we say that $(x,y)$ is an {\it F-chaotic pair} in $(G,X)$ if there can be found
sequences $\{l_n\}_1^{\infty}$, $\{r_n\}_1^{\infty}$,
$\{s_n\}_1^{\infty}$ and $\{t_n\}_1^{\infty}$ in $G$ such that
\begin{gather*}
\lim_{n\to+\infty}d(l_n x,y)=0,\quad \lim_{n\to+\infty}d(r_n x,y)>0\intertext{and}
\lim_{n\to+\infty}d(s_n x,s_n y)=0,\quad \lim_{n\to+\infty}d(t_n x,t_ny)>0.
\end{gather*}
\end{itemize}

In this paper, applying Lemma~\ref{lem1.2},
we will show that if ${\mathcal{C}}_{\mathcal{F}}(x)$ is not S-generic,
then the chaotic behavior occurs near ${\mathcal{C}}_{\mathcal{F}}(x)$; see
Theorems~\ref{thm3.2} and~\ref{thm3.4} in $\S\ref{sec3.1}$.
On the other hand, whenever ${\mathcal{C}}_{\mathcal{F}}(x)$ is S-generic and if it is non-minimal,
then chaotic dynamics exhibits in ${\mathcal{C}}_{\mathcal{F}}(x)$;
that is the following which is a consequence of Theorem~\ref{thm:not-generic-LY} in $\S\ref{sec3.2}$.

\begin{thm}\label{thm1.3}%%%
If ${\mathcal{C}}_{\mathcal{F}}(x)$ is S-generic and itself is not a minimal
subset of $(G,X)$, then there exists an F-chaotic chair in ${\mathcal{C}}_{\mathcal{F}}(x)$.
\end{thm}

Moreover, we shall study in $\S\ref{sec3.2}$ the Auslander-Yorke chaotic dynamics for any non-minimal $G$-system as follows:

\begin{thm}\label{thm1.4}%%%
If $X$ is S-generic and not minimal, then $(G,X)$ is point transitive and one can find an $\epsilon>0$ such that
for any $\hat{x}\in X$, there exists a dense subset $S_\epsilon(\hat{x})$ of $X$
satisfying that for each $y\in S_\epsilon(\hat{x})$ there is a sequence $\{t_n\}_1^\infty$ in $G$
so that $\lim_{n\to+\infty}d(t_n x,t_n y)\geq\epsilon$.
\end{thm}

Here ``point transitive'' and the following ``almost periodic point'' will be precisely defined in $\S\ref{sec2.2}$. Next we will further investigate the so-called $2$-sensitivity near ${\mathcal{C}}_{\mathcal{F}}(x)$
for any non-minimal $G$-system in $\S\ref{sec4}$.

\begin{thm}\label{thm1.5}%%%
If ${\mathcal{C}}_{\mathcal{F}}(x)$ is S-generic non-minimal
and if almost periodic points of $(G,X)$ are dense inside ${\mathcal{C}}_{\mathcal{F}}(x)$,
then one can find two distinct points $x_1,x_2\in {\mathcal{C}}_{\mathcal{F}}(x)$
such that for any $\hat{x}\in {\mathcal{C}}_{\mathcal{F}}(x),
U\in \mathscr{U}_{\hat{x}}$, $U_1\in \mathscr{U}_{x_1}$, and $U_2\in \mathscr{U}_{x_2}$,
there exist $y_1,y_2\in U$ and $g\in G$ with
$gy_1\in U_1$ and $gy_2\in U_2$.
\end{thm}

Furthermore, if $G$ is commutative, then we can obtain a more stronger sensitivity to initial conditions as follows:

\begin{thm}\label{thm1.6}%%%
Let $G$ be abelian. If ${\mathcal{C}}_{\mathcal{F}}(x)$ is S-generic non-minimal
and if almost periodic points of $(G,X)$ are dense inside ${\mathcal{C}}_{\mathcal{F}}(x)$,
then one can find an $\infty$-countable subset $K$ of ${\mathcal{C}}_{\mathcal{F}}(x)$
such that for any $\hat{x}\in {\mathcal{C}}_{\mathcal{F}}(x)$,
any $n$ distinct points $x_1,\ldots,x_n\in K$
with $n\geq 2$ and any $U\in \mathscr{U}_{\hat{x}}$, $U_i\in \mathscr{U}_{x_i}$,
there exist $n$ points $y_1,\ldots,y_n\in U$ and some $g\in G$ with
$gy_i\in U_i$ for all $1\leq i\leq n$.
\end{thm}

Although chaotic behavior possibly occurs \textit{near} a non-S-generic minimal center of attraction by Theorems~\ref{thm3.2} and \ref{thm3.4} in $\S\ref{sec3.1}$, yet we will construct an example in $\S\ref{sec5.1}$ to show that there might exist no chaotic dynamics \textit{in} a non-S-generic minimal center of attraction.

Since $\mathbb{R}_+$ and $\mathbb{Z}$ are commutative and so they are amenable under the discrete topology (cf.~\cite{Argabright,Hindman-06}), hence Theorems~\ref{thm1.3}, \ref{thm1.4}, \ref{thm1.5} and \ref{thm1.6} generalize the recent works \cite{X, zh-ye, Dai-2} for $\mathbb{Z}$- and $\mathbb{R}_+$-systems.

Finally, it should be noted that in general different F{\o}lner nets in $G$ may define different minimal centers of attraction of a same point of $X$; see Example~\ref{sec5.3} below.

%%%%%%%%%%%%%%%%%%%%%%%%%%%%%%%%%%%%%%
\section{Preliminaries}\label{sec2}%%%
In this section we will introduce some preliminaries needed in our discussion later on.

\subsection{Sets in a group}
Let $(G,\cdot)$ be a discrete group.
An {\it idempotent}
$t$ in $G$ is an element satisfying $t\cdot t=t$.
A subset $I$ of $G$ is called a {\it left ideal} of $G$ if $GI\subseteq I$,
a {\it right ideal} if $IG\subseteq I$, and a {\it two sided ideal} (or simply an ideal ) if it is
both a left and right ideal. A {\it minimal left ideal} is the left ideal that
does not contain any proper left ideal. Similarly, we can define {\it minimal
right ideal} and {\it minimal ideal}. An element in $G$ is called a {\it minimal
idempotent} if it is an idempotent in some minimal left ideal of $G$.
Each left ideal of every compact Hausdorff right-topological semigroup
contains some minimal left ideal and every minimal left ideal has an idempotent;
see, e.g.,~\cite{F-81,Hindman}.

A {\it filter} on $G$ is a nonempty collection $\mathcal {S}$ of subsets of $G$ with properties:
(1) if $A,B\in\mathcal {S}$, then $A\cap B\in\mathcal {S}$;
(2) if $A\in\mathcal {S}$ and $A\subset B\subset G$, then $B\in \mathcal {S}$;
(3) $\emptyset\notin \mathcal {S}$.
An {\it ultrafilter} on $G$ is a filter on $G$ which is not properly contained in any other filter on $G$.
We take the points of the
Stone-\v{C}ech compactification $\beta G$ of $G$ to be all the ultrafilters on $G$;
see e.g.~\cite{Hindman}.
Since $(G, \cdot)$ is a discrete group,
we can extend its operation $\cdot$ to $\beta G$ such that
$(\beta G, \cdot)$ is a compact Hausdorff right-topological semigroup; see e.g.~\cite[Theorem 4.1]{Hindman}.

By $K_{\beta G}$ we denote the unique minimal ideal of the compact Hausdorff semigroup $(\beta G,\cdot)$. If $A$ is a subset of $G$,
then $c\ell_{\beta G} A=\{p\in\beta G\,|\,A \in p\}$
is a base clopen subset of $\beta G$; see e.g.~\cite[Theorem 3.18]{Hindman}.
A subset $P\subseteq G$ is  called {\it syndetic}
if there is a finite set $F\subset G$ such that
$F^{-1}P={\bigcup}_{g \in F}g^{-1}P=G$;
it is {\it thick}
if for every finite subset $A\subseteq G$ there is some $t\in G$
such that $P\supseteq At$.
It is well known that
$P$ is syndetic if and only if $P$ intersects every thick set;
if $P\subseteq G$ is syndetic, then $tP$ is also syndetic for every $t\in G$;
 see, e.g.,~\cite{F-81,BHM,BM}.

\begin{lem}[{\cite[Theorem 2.9]{BHM}}]\label{thick-equi}
A subset $A\subseteq G$ is thick if and only if
$c\ell_{\beta G}A$ contains a left ideal of $\beta G$.
\end{lem}

A subset $P\subseteq G$ is  called {\it piecewise syndetic}
if there is a finite subset $F$ of $G$ satisfying that for every finite subset
$A$ of $G$ there is some $g\in G$ such that $Ag\subseteq F^{-1}P$;
it is {\it thickly syndetic} if for every finite subset
$A\subseteq G$ there is a syndetic set $Q \subseteq G$ such that $AQ\subseteq P$.
If $P\subseteq G$ is piecewise syndetic, then so is $tP$ for every
$t\in G$; see~\cite[Theorem 2.3]{BHM}.

\begin{lem}[{\cite[Theorem 4.40]{Hindman}}]\label{lem2.2}%%%
$S\subseteq G$ is piecewise syndetic iff $K_{\beta G}\cap c\ell_{\beta G}S\not=\emptyset$.
\end{lem}

Then this implies the following

\begin{cor}[{\cite[Theorem~1.24]{F-81}}]
Let $S=P_1\cup\dotsm\cup P_q$ be a finite partition of a piecewise syndetic subset $S$ of $G$. Then one of the cells $P_j$ is piecewise syndetic.
\end{cor}

\subsection{Topological dynamics}\label{sec2.2}%%%
By a {\it G-system} $(G,X)$ we here mean that $X$ is a compact
metric space, and $(G,\cdot)$ is a discrete infinite amenable group such that $G$ consists of continuous transformations of $X$ with $e(x)=x$ and $gf(x)=g(f(x))$ for all $x\in X$ and $f,g\in G$. Here $e$ is the identity of $G$.
For any $x\in X$ and any subsets $U, V$ of $X$,
we write
\[N(U,V)=\{g\in G\,|\,U\cap g^{-1}V\neq\varnothing\}\quad \textrm{and}\quad
N(x,U)=\{g\in G\,|\,gx\in U\}.\]

A $G$-system $(G,X)$ is called {\it point transitive}
if there exists a point $y\in X$ such that
$Gy$ is dense in $X$ and such point $y$ is called a {\it transitive point};
it is {\it minimal} if $c\ell_X{Gx}=X\ \forall x\in X$.
A point $x\in X$ is called {\it minimal} if $c\ell_X{Gx}$ is minimal under $G$;
it is {\it almost periodic} (or {\it uniformly recurrent} in some literature like \cite{F-81,CD}) if $N(x,U)$ is syndetic in $G$ for every $U\in\mathscr{U}_x$.\\

\begin{prop}[{\cite[Proposition 3.2]{KM}}]\label{prop:point-tran-equi-trans}
Any $G$-system $(G,X)$ has a dense $G_\delta$-set of transitive points if and only if
it is point transitive.
\end{prop}

Notice that since in our setting the phase space $X$ is compact metric, the point transitive is equivalent to the topologically transitive.

For any $p\in \beta G$, we call $y\in X$ a {\it $p$-limit point}
of $x\in X$ if
$y=p\textrm{-}\lim_{g\in G}gx$; i.e., for all $U\in\mathscr{U}_y$,
$\{g\in G\,|\,gx\in U\}\in p$.
It is well known that for every $p\in \beta G$ and every $x\in X$,
$p\textrm{-}\lim_{g\in G}gx$ exists uniquely; see, e.g.,~\cite[Theorem 3.48]{Hindman}.

\begin{lem}[{\cite[Theorem 19.23]{Hindman}}]\label{mini-point-equ}
Let $(G,X)$ be a $G$-system and $x\in X$.
Then the followings are pairwise equivalent.
\begin{enumerate}
\item $x$ is a minimal point.
\item $x$ is almost periodic.
\item There exist some $y\in X$ and a minimal idempotent $p$ in $\beta G$ such that $p\textrm{-}\lim_{g\in G}gy=x$ (i.e. $py=x$).
\end{enumerate}
\end{lem}

\subsection{Characterization of minimal center of attraction}\label{sec2.3}
Let $(G,X)$ be a $G$-system and let $\mathcal{F}=\{F_n\}_{n\in D}$ be any F{\o}lner net in $G$.
We now will prove Lemma~\ref{lem1.2} stated in $\S\ref{sec1}$.

\begin{proof}[Proof of Lemma~\ref{lem1.2}]
Relative to the F{\o}lner net $\mathcal {F}=\{F_n\}_{n\in D}$ in $G$, for any point $x\in X$ we write
\[I(x)=\left\{y\in X\,|\,\overline{d}_{\mathcal{F}}(N(x,U))>0\ \forall U\in \mathscr{U}_y\right\}.\]
We first claim that ${\mathcal{C}}_{\mathcal{F}}(x)\subseteq I(x)$.
Indeed, for any $z\in {\mathcal{C}}_{\mathcal{F}}(x)$ and $U\in \mathscr{U}_z$, we have
$\overline{d}_{\mathcal{F}}(N(x,U))>0$;
otherwise, there would exist some $\epsilon>0$ such that
$\overline{d}_{\mathcal{F}}(N(x,B_{3\epsilon}(z)))=0$
which implies that
\[\lim_{n\in D}\frac{|N(x,B_{3\epsilon}(z))\cap F_n|}{|F_n|}=0.\]
Further ${\mathcal{C}}_{\mathcal{F}}(x)\setminus B_{2\epsilon}(z)$ is also an $\mathcal{F}$-center of attraction of $x$ for $(G,X)$.
And we thus arrive at a contradiction to the minimality
of ${\mathcal{C}}_{\mathcal{F}}(x)$.

Now, it is left to prove that ${\mathcal{C}}_{\mathcal{F}}(x)\supseteq I(x)$.
On the contrary, assume $z\in I(x)\setminus {\mathcal{C}}_{\mathcal{F}}(x)$ and then there exists $ \epsilon> 0$ such that
$d(z, {\mathcal{C}}_{\mathcal{F}}(x))\geq 3\epsilon$.
It is clear that
\[N(x, B_{\epsilon}(z))\cap N(x, B_{\epsilon}({\mathcal{C}}_{\mathcal{F}}(x)))=\emptyset.\]
However, by definitions, we have that
\begin{equation*}
\overline{d}_{\mathcal{F}}(N(x, B_{\epsilon}(z)))>0\quad
\textrm{and}\quad
d_{\mathcal{F}}(N(x,B({\mathcal{C}}_{\mathcal{F}}(x),\epsilon)))= 1.
\end{equation*}
This is a contradiction.

The proof of Lemma~\ref{lem1.2} is thus completed.
\end{proof}

To obtain some useful properties
of minimal center of attraction of $(G,X)$, we need some basic facts about F{\o}lner net in $G$.

\begin{lem}\label{lem2.6}
There exists a subnet $\{F_{n_k}\}_{k\in E}$ of $\mathcal{F}$ in $G$ such that $\lim_{k\in E}|F_{n_k}|=\infty$.
\end{lem}

\begin{proof}
Let $E=D\times \mathbb{N}$ and direct $E$ by agreeing that $(m,k)\leq (m',k')$ if and only if $m\leq m'$ and $k\leq k'$.
We claim that for each $k\in\mathbb{N}$  and each $m\in D$
there exists $n(m,k)\in D$ with $n(m,k)\geq m$ such that $|F_{n(m,k)}|>k$
and then the net $\{F_{n(m,k)}\}_{(m,k)\in E}$ is the required.
Otherwise, there exist $M\in \mathbb{N}$ and $m\in D$ such that $|F_{n}|\leq M$ for all $n\geq m$.
Choose $M+1$ distinct points $g_1,g_2,\ldots,g_{M+1}$ in $G$.
Since $\{F_n\}_{n\in D}$ is a F{\o}lner net in $G$,
it follows that
\[\lim_{n\in D}\frac{|g_iF_n\vartriangle F_n|}{|F_n|}=0, \]
for all $i=1,2,\ldots,M+1$, which implies that there exists $m'\in D$ such that
\[\frac{|g_iF_n\vartriangle F_n|}{|F_n|}<\frac{1}{M+1}\]
for all $n\geq m'$ and $i=1,2,\ldots,M+1$.
Thus there exists $n\in D$ with $n\geq m$ and $n\ge m'$ such that
${|g_iF_n\vartriangle F_n|}/{|F_n|}<{1}/(M+1)$
for $i=1,2,\ldots,M+1$.
This implies that $g_iF_n=F_n$ for $i=1,2,\ldots,M+1$ as $|F_{n}|\leq M$.
By Pigeon house lemma, there exist $i,j$ with $i\neq j$ such that $g_i=g_j$.
This is a contradiction and so it completes the proof of Lemma~\ref{lem2.6}.
\end{proof}

The definition of upper
density of a subset of $G$ is sometimes not convenient to handle.
We provide the following property for it.

\begin{prop}\label{prop:sup-subnet}
Let $A$ be an arbitrary subset of $G$.
Then there exists a subnet $\{F_{n_k}\}_{k\in E}$ of $\mathcal{F}$ in $G$ such that
\begin{equation*}
\overline{d}_{\mathcal {F}}(A)=\lim_{k\in E}\frac{|F_{n_k}\cap A|}{|F_{n_k}|}.
\end{equation*}
\end{prop}

\begin{proof}
Suppose that
$\overline{d}_{\mathcal {F}}(A)=\alpha$.
Let $E=D\times \mathbb{N}$ and direct $E$ by agreeing that
 $(m,k)\leq (m',k')$ if and only if $m\leq m'$ and $k\leq k'$.
Fix any $k\in\mathbb{N}$,
there exists $m'\in D$ such that for every $n\in D$ with $n\geq m'$, there holds
\[(\alpha+\frac{1}{k})|F_{n}|>|F_{n}\cap A|;\]
and for each $m\in D$,
we can choose $n(m,k)\in D$ with $n(m,k)\geq m,m'$ both, such that
\[(\alpha+\frac{1}{k})|F_{n(m,k)}|>|F_{n(m,k)}\cap A|>(\alpha-\frac{1}{k})|F_{n(m,k)}|.\]
Then  $\{F_{n(m,k)}\}_{(m,k)\in E}$ is a subnet of $\{F_n\}_{n\in D}$ and satisfies that
\[\lim_{(m,k)\in E}\frac{|F_{n(m,k)}\cap A|}{|F_{n(m,k)}|}=\alpha.\]
This proves Proposition~\ref{prop:sup-subnet}.
\end{proof}

Given any $g\in G$, let $\mathcal{F}g=\{F_ng\}_{n\in D}$. Clearly, it is also a F{\o}lner net in $G$. As a result of Lemma~\ref{lem1.2} and Proposition~\ref{prop:sup-subnet}, we can obtain the following

\begin{cor}\label{cor2.8}
For any $x\in X$, ${\mathcal{C}}_{\mathcal{F}}(x)$ is $G$-invariant; moreover, ${\mathcal{C}}_{\mathcal{F}}(x)$=${\mathcal{C}}_{\mathcal{F}g}(gx)$ for any $g\in G$. Consequently, if $G$ is commutative, then ${\mathcal{C}}_{\mathcal{F}}(x)={\mathcal{C}}_{\mathcal{F}}(gx)$ for all $g\in G$.
\end{cor}

\begin{proof}
Let $y\in {\mathcal{C}}_{\mathcal{F}}(x)$ and $g\in G$.
We want to show that $gy\in {\mathcal{C}}_{\mathcal{F}}(x)$.
Let $U\in \mathscr{U}_{gy}$.
Then $g^{-1}U\in \mathscr{U}_{y}$.
By Lemma~\ref{lem1.2} and Proposition~\ref{prop:sup-subnet},
there exists a subnet $\{F_{n_k}\}_{k\in E}$ of $\mathcal{F}$ such that
\begin{align*}
\lim_{k\in E}\frac{|N(x,U)\cap F_{n_k}|}{|F_{n_k}|}
&=\lim_{k\in E}\frac{|N(x,U)\cap gF_{n_k}|}{|F_{n_k}|}\\
&=\lim_{k\in E}\frac{|N(x,g^{-1}U)\cap F_{n_k}|}{|F_{n_k}|}\\
&>0.
\end{align*}
Therefore ${\mathcal{C}}_{\mathcal{F}}(x)$ is $G$-invariant. The second part is obvious.
\end{proof}

By Corollary~\ref{cor2.8}, it follows that ${\mathcal{C}}_{\mathcal{F}}(x)$
can be viewed as a subsystem of $(G,X)$.

\begin{cor}\label{cor2.9}
For any $x\in X$, it holds that ${\mathcal{C}}_{\mathcal{F}}(x)\subseteq c\ell_X{Gx}$. Moreover, if $x\in {\mathcal{C}}_{\mathcal{F}}(x)$,
then $c\ell_X{Gx}={\mathcal{C}}_{\mathcal{F}}(x)$.
\end{cor}

\begin{proof}
The first part is trivial. Now suppose $x\in {\mathcal{C}}_{\mathcal{F}}(x)$.
By Corollary~\ref{cor2.8},
it follows that $c\ell_X{Gx}\subseteq {\mathcal{C}}_{\mathcal{F}}(x)$.
It is easy to see that for every open neighborhood $U$ of $c\ell_X{Gx}$,
$d_{\mathcal{F}}(N(x,U))=1$.
From the minimality of ${\mathcal{C}}_{\mathcal{F}}(x)$,
we can conclude that $c\ell_X{Gx}={\mathcal{C}}_{\mathcal{F}}(x)$.
\end{proof}

Given any two points $x, y$ in $X$, the pair $\{x, y\}$ is
called {\it asymptotic}  if  for every $n\in \mathbb{N}$, there exists finite subset $F$  of $G$
such that $d(gx,gy)<1/n$  for all $g\in G$ with $g\not\in F$.

\begin{prop}\label{prop:MCA-approx}
For any two points $x, y\in X$,
if $\{x,y\}$ is asymptotic for $(G,X)$,
then ${\mathcal{C}}_{\mathcal{F}}(x)={\mathcal{C}}_{\mathcal{F}}(y)$.
\end{prop}

\begin{proof}
Assume that $\{x,y\}$ is asymptotic.
Fix an arbitrary point $z\in {\mathcal{C}}_{\mathcal{F}}(x)$.
Let $U\in \mathscr{U}_z$.
Choose an $\epsilon>0$ with $B_\epsilon(z)\subset U$.
Since $\{x,y\}$ is asymptotic,
it follows that there exists  a finite subset $F$ of $G$ such that
 $d(gx,gy)<{\epsilon}/{2}$  for all $g\in G\setminus F$.
 Then $N(y,U)\supset N(x,B_{{\epsilon}/{2}}(z))\setminus F$.
By Lemma~\ref{lem1.2},
it follows that
$\overline{d}_{\mathcal {F}}(N(x,B_{{\epsilon}/{2}}(z))):=\alpha>0$.
From Proposition~\ref{prop:sup-subnet},
there exists a subnet $\{F_{n_k}\}_{k\in E}$ of $\mathcal{F}$ such that
\[\lim_{k\in E}\frac{|F_{n_k}\cap N(x,B_{{\epsilon}/{2}}(z))|}{|F_{n_k}|}=\overline{d}_{\mathcal {F}}(N(x,B_{{\epsilon}/{2}}(z))).\]
Moreover, by Lemma~\ref{lem2.6},
we can require that $\lim_{k\in D}|F_{n_k}|=+\infty$.
Then
\[\lim_{k\in E}\frac{|F_{n_k}\cap (N(x,B_{{\epsilon}/{2}}(z))\setminus F)|}{|F_{n_k}|}=\alpha.\]
Thus
\begin{align*}
\overline{d}_{\mathcal {F}}(N(y,U))\geq& \overline{d}_{\mathcal {F}}(N(x,B_{{\epsilon}/{2}}(z))\setminus F)>0.
\end{align*}
Since $U$ is arbitrary, it follows that $z\in {\mathcal{C}}_{\mathcal{F}}(x)$ and then
${\mathcal{C}}_{\mathcal{F}}(x)\subseteq {\mathcal{C}}_{\mathcal{F}}(y)$.
Similarly, we have ${\mathcal{C}}_{\mathcal{F}}(y)\subseteq {\mathcal{C}}_{\mathcal{F}}(x)$.
This thus completes the proof of Proposition~\ref{prop:MCA-approx}.
\end{proof}

The following Proposition is a folklore result, we provide a proof for completeness.

\begin{prop}\label{prop:invar-thick}
If $\Lambda$ is a nonempty $G$-invariant closed subset of $c\ell_X{Gx}$,
then for any open neighborhood $U$ of $\Lambda$,
$N(x,U)$ is thick in $G$.
\end{prop}

\begin{proof}
Suppose that $\Lambda$ is a nonempty $G$-invariant closed subset of $c\ell_X{Gx}$.
Let $U$ be any open neighborhood of $\Lambda$.
Let $F$ be any finite subset of $G$.
For every $y\in\Lambda$  and every $g\in F$,
there exists $V_g\in \mathscr{U}_y$ such that
 $gV_g\subset U$ since $\Lambda$ is $G$-invariant.
Put $V_y:=\bigcap_{g\in F}V_g$.
Then $V_y\in \mathscr{U}_y$ and $FV_y\subset U$.
Take $V=\bigcup\{V_y\,|\,y\in\Lambda\}$.
Then $V$ is an open neighborhood of
$\Lambda$ and $FV \subset U$.
Since $\Lambda\subset c\ell_X{Gx}$,
it follows that
there exists $t\in G$ such that $tx\in V$. Hence
$Ftx\subset U$, which implies that
$N(x,U)\supset Ft$.
This ends the proof of Proposition~\ref{prop:invar-thick}.
\end{proof}

%%%%%%%%%%%%%%%%%%%%%%%%%%%%%%%%%%%%%%%%%%%%%%%%%%%%
\section{Li-Yorke chaotic pairs and sensitivity (I)}\label{sec3}
Let $(G,X)$ be a $G$-system in the sequel of this section.
Recall that for any two points $x, y$ in $X$, the pair $\{x, y\}$ is
called
\begin{itemize}
\item  {\it proximal} if there exists a sequence $\{t_n\}_1^\infty$
in $G$ such that $\lim_{n\to+\infty}d(t_n x, t_n y)=0$ (cf.~\cite[Definition~8.1]{F-81}); and
\item  {\it Li-Yorke chaotic} if $\{x,y\}$ is proximal but not asymptotically for $(G,X)$.
\end{itemize}
If there can be found an uncountable
subset $C$ of $X$ such that for any $x, y\in C, x\not=y$, $\{x,y\}$
is a Li-Yorke chaotic pair for $(G,X)$, then we say $(G,X)$ is
{\it Li-Yorke chaotic}~\cite{DT}; see, e.g., Li-Yorke~\cite{LY-75} for $\mathbb{Z}_+$-systems.

First of all, the proximal pair can be characterized via ultrafilter as follows.

\begin{lem}[{\cite[Lemma 19.22]{Hindman}}]\label{prox-point-equ}
Given any $x,y\in X$, the pair $\{x,y\}$ is proximal if and only if there exists $p\in \beta G$ such that
$p\textrm{-}\lim_{g\in G}gx=p\textrm{-}\lim_{g\in G}gy$ (i.e. $px=py$).
\end{lem}

Another customary description of chaos is {\it sensitivity} to initial conditions (cf., e.g., \cite{Devaney-89, LY}
for $\mathbb{Z}_+$-systems and \cite{Dai-1, Dai-2} for $\mathbb{R}_+$-systems):
\begin{itemize}
\item There exists an $\epsilon>0$ such that for all $x\in X$ and any $U\in\mathscr{U}_x$,
there are some $y\in U$ and some $g\in G$ with
$d(gx, gy)\geq\epsilon$; such an $\epsilon$ is called a {\it sensitive constant} of $(G,X)$.
\end{itemize}
 See, e.g., \cite{KM,DT}.

\subsection{Non-S-generic case}\label{sec3.1}
Let $x\in X$ and $\mathcal{F}=\{F_n\}_{n\in D}$ be a F{\o}lner net in $G$.
We first show that if ${\mathcal{C}}_{\mathcal{F}}(x)$ has no the S-generic dynamics,
then $(G,X)$ has the
chaotic dynamics near ${\mathcal{C}}_{\mathcal{F}}(x)$.

\begin{thm}\label{thm3.2}
If ${\mathcal{C}}_{\mathcal{F}}(x)$ is not S-generic,
then one can find some point $y\in\Lambda$, for any closed
$G$-invariant subset $\Lambda\subseteq {\mathcal{C}}_{\mathcal{F}}(x)$,
 such that $\{x, y\}$ is Li-Yorke chaotic for $(G,X)$.
\end{thm}

\begin{proof}
Given any $x$ in $X$, let ${\mathcal{C}}_{\mathcal{F}}(x)$ be not S-generic.
Let $\Lambda$ be a
nonempty $G$-invariant closed subset of ${\mathcal{C}}_{\mathcal{F}}(x)$.
Then by Corollary~\ref{cor2.9} it follows that $x\notin {\mathcal{C}}_{\mathcal{F}}(x)$.
From Zorn's lemma, there exists a minimal set $\Lambda'$ in $\Lambda$.
From Lemma~\ref{lem1.2},
it follows that $x$ is proximal to $\Lambda'$; that is,
there exists a sequence $\{t_n\}_1^\infty$ in $G$ such that
\[\lim_{n\to+\infty}d(t_n x,\Lambda')=0.\]
This implies that $\Lambda'\cap\beta Gx\neq\emptyset$.
Take
\[F=\{p\in\beta G\,|\, p\textrm{-}{\lim}_{g\in G}gx\in\Lambda'\}.\]
Since $\Lambda'$ is a compact
$G$-invariant subset, $F$ is a compact Hausdorff semigroup.
From the Numakura-Wallace-Ellis theorem (see, e.g.,~\cite[Theorem 1.18]{Ellis-semigroup},~\cite[Lemma~8.4]{F-81} and \cite[Theorem 2.5]{Hindman}),
it follows that there exists an idempotent $u\in F$.
Let $y=u\textrm{-}{\lim}_{g\in G}gx$; then
\[u\textrm{-}{\lim}_{g\in G}gx=u\textrm{-}{\lim}_{g\in G}gy.\]
Therefore,
by Lemma~\ref{prox-point-equ}, it follows that the pair $\{x, y\}$ is proximal for $(G,X)$.

On the other hand,
we  claim that $\{x, y\}$ is not asymptotic for $(G,X)$.
Otherwise by Proposition~\ref{prop:MCA-approx}, there holds
${\mathcal{C}}_{\mathcal{F}}(x)={\mathcal{C}}_{\mathcal{F}}(y)$
and then ${\mathcal{C}}_{\mathcal{F}}(x)$ is S-generic.
This is a contradiction.

The proof of Theorem~\ref{thm3.2} is therefore completed.
\end{proof}

We remark that the point $x$ does not need to be recurrent in Theorem~\ref{thm3.2}.

\begin{cor}\label{cor:non-gen-cen-set}
If ${\mathcal{C}}_{\mathcal{F}}(x)$ is not S-generic,
 then there exists a point $y\in {\mathcal{C}}_{\mathcal{F}}(x)$ such that
the pair $\{x, y\}$ is Li-Yorke chaotic for $(G,X)$ and the set
$N(x, B_\epsilon(y))$
is a central set in $G$, which has positive upper density.
\end{cor}

\begin{proof}
Let $\Lambda\subset {\mathcal{C}}_{\mathcal{F}}(x)$ be a minimal set.
Then there exists by Theorem~\ref{thm3.2} a point
$y\in\Lambda$ such that the pair $\{x, y\}$ is
Li-Yorke chaotic for $(G,X)$.
By~\cite[Definition 19.20]{Hindman-06},
it follows that for each $\epsilon > 0$ the set $N(x,B_\epsilon(y))$ is a central set of $G$.
In addition, from Lemma~\ref{lem1.2},
we can conclude that $N(x, B_\epsilon(y))$ has positive upper density.
This proves
the Corollary~\ref{cor:non-gen-cen-set}.
\end{proof}

Motivated by~\cite{Banks-sensitive,Dai-1,G93}
we can obtain the following theorem that captures sensitivity near the minimal
center of attraction of $x$ for the $G$-system $(G,X)$.

\begin{thm}\label{thm3.4}
Let ${\mathcal{C}}_{\mathcal{F}}(x)$ be not S-generic.
If the almost periodic points of $(G,X)$ are dense in ${\mathcal{C}}_{\mathcal{F}}(x)$,
then $(G,X)$ is sensitive near ${\mathcal{C}}_{\mathcal{F}}(x)$ in the following sense:
one can find an $\epsilon>0$
such that for any points $a\in X$, $\hat{x}\in {\mathcal{C}}_{\mathcal{F}}(x)$ and any $U\in \mathscr{U}_{\hat{x}}$,
there exist $y\in U$ and some $g\in G$ with
$d(ga, gy)\geq \epsilon$.
\end{thm}

\begin{proof}
Since ${\mathcal{C}}_{\mathcal{F}}(x)$ is not S-generic,
it follows that ${\mathcal{C}}_{\mathcal{F}}(x)$ is not minimal itself and
so it contains two distinct minimal points $z_1,z_2$ with $c\ell_X{Gz_1}\neq c\ell_X{Gz_2}$.
Take
$d(c\ell_X{Gz_1},c\ell_X{Gz_2})=3\delta$,
where $d(A, B) = \inf\{d(a, b)\,|\,a\in A, b\in B \}$ for any subsets $A$, $B$ of $X$.
Thus $\delta > 0$ such that
for all $\hat{x}\in {\mathcal{C}}_{\mathcal{F}}(x)$ there exists a corresponding orbit $Gz$ in ${\mathcal{C}}_{\mathcal{F}}(x)$,
not necessarily recurrent, such that
$d(x,c\ell_X{Gz})\geq\delta$
where $d(\hat{x}, A) = \inf\{d(\hat{x}, a)\,|\, a\in A \}$ for any subset $A$ of $X$.
We will show that $(G,X)$ is sensitive with
sensitive constant $\epsilon={\delta}/{4}$.

For this, we let $\hat{x}$ be an arbitrary point in ${\mathcal{C}}_{\mathcal{F}}(x)$
and $U$ an open neighborhood of $\hat{x}$ in $X$.
Since the collection of all almost periodic points of $(G,X)$ are dense in ${\mathcal{C}}_{\mathcal{F}}(x)$,
 there exists an almost periodic point $p\in X$ with
\[p\in  U\cap B_{{\epsilon}/{2}}(\hat{x})\cap {\mathcal{C}}_{\mathcal{F}}(x).\]
As we noted above,
there must exists
another point $z\in {\mathcal{C}}_{\mathcal{F}}(x)$ such that the orbit $c\ell_X{Gz}$ is of distance
at least $4\epsilon$ from the given point $\hat{x}$.

Since $p$ is an almost periodic point,
it follows that there exists a finite subset $F$ of $G$ such that $G=F^{-1}N(p, B_{{\epsilon}/{2}}(p))$.
As ${\mathcal{C}}_{\mathcal{F}}(x)$ is the minimal $\mathcal{F}$-center of attraction of $x$,
by Lemma~\ref{lem1.2},
there exists $t\in G$ such that $tx\in U$.
From Proposition~\ref{prop:invar-thick},
it follows that there exists $r\in G$ such that
$Ftr\subset N(x,B_\epsilon(c\ell_X{Gz}))$,
that is
$Ftrt^{-1}(tx)\subset B_\epsilon(c\ell_X{Gz})$.
It is clear that $F'trt^{-1}p\subset B_{{\epsilon}/{2}}(p)$ for some nonempty subset $F'$ of $F$.
Take $y=tx$.
Then $y\in U$ and
for every $g\in F'trt^{-1}$,  there holds
$d(gp,gy)\geq 2\epsilon$.

Now for any $a\in X$, using the triangle inequality,
we can conclude that either
$d(gp,ga)\geq \epsilon$
or
$d(gy,ga)\geq \epsilon$
for every $g\in F'trt^{-1}$.

Since $\hat{x}$, $U$ both are arbitrary and $y\in U$,
the proof of Theorem~\ref{thm3.4} is completed.
\end{proof}

It should be noticed here that the subsystem $(G,{\mathcal{C}}_{\mathcal{F}}(x))$ need not be point transitive comparing with \cite[Proposition~2.5]{DT}.

\subsection{S-generic case}\label{sec3.2}
In this subsection, we will show that the minimal center
of attraction exhibits more complicate behavior if it is S-generic. Let $x\in X$ and $\mathcal{F}=\{F_n\}_{n\in D}$ a F{\o}lner net in $G$.

\begin{thm}\label{thm:not-generic-LY}
Let ${\mathcal{C}}_{\mathcal{F}}(x)$ be an S-generic non-minimal subset of $(G,X)$ and
$S_x = \{y\in {\mathcal{C}}_{\mathcal{F}}(x)\,|\, {\mathcal{C}}_{\mathcal{F}}(x) = {\mathcal{C}}_{\mathcal{F}}(y)\}$.
Then for any $y\in S_x$ and any minimal subset $\Delta\subset {\mathcal{C}}_{\mathcal{F}}(x)$, there exists $z\in \Delta$
with the properties that $\{y, z\}$ is a Li-Yorke chaotic pair for $(G,X)$ and such that
there exist two sequences $\{t_n\}_1^\infty$ and $\{s_n\}_1^\infty$ in $G$ so that
\[\lim_{n\to+\infty}d(t_n y,z)=0\quad \text{and}\quad \lim_{n\to+\infty}d(t_n y,z)\ge\frac{1}{2}\text{diam}({\mathcal{C}}_{\mathcal{F}}(x))\]
where $\text{diam}({\mathcal{C}}_{\mathcal{F}}(x))$ stands for the diameter of the set ${\mathcal{C}}_{\mathcal{F}}(x)$.
Moreover, if $G$ is commutative, then the set $S_x$ is dense in ${\mathcal{C}}_{\mathcal{F}}(x)$.
\end{thm}

\begin{proof}
Since ${\mathcal{C}}_{\mathcal{F}}(x)$ is S-generic,
it follows that  there exists $y\in {\mathcal{C}}_{\mathcal{F}}(x)$ with
${\mathcal{C}}_{\mathcal{F}}(y) = {\mathcal{C}}_{\mathcal{F}}(x)$.
Because ${\mathcal{C}}_{\mathcal{F}}(x)$ is not a minimal subset of $(G,X)$,
$c\ell_X{Gz}$ is not minimal for each $z\in S_x$ by Corollary~\ref{cor2.9}.

Let $\Delta$ be a minimal subset of ${\mathcal{C}}_{\mathcal{F}}(x)$ for $(G,X)$.
Then each $z\in S_x$ is proximal to $\Delta$ following by the proof of  Theorem~\ref{thm3.2}.
Moreover
by  ~\cite[Proposition 8.6]{F-81},
it follows that for every $z\in S_x$ ,
there corresponds some point
$y\in \Delta$ such that $\{z,y\}$ is proximal and $y$ is almost periodic.
Clearly, the pair $\{z, y\}$ is a Li-Yorke chaotic for $(G,X)$.
Furthermore, from Lemma~\ref{lem1.2},
we can conclude that there exist two sequences $\{t_n\}_1^\infty$ and $\{s_n\}_1^\infty$ in $G$ such that
\[\lim_{n\to+\infty}d(t_n z, y)
\geq\frac{1}{2}\text{diam}({\mathcal{C}}_{\mathcal{F}}(x))\quad \text{and}\quad \lim_{n\to+\infty}d(s_n z, y)=0.\]
In addition, if $G$ is commutative, then by Corollary~\ref{cor2.8} we see that $S_x$ is dense in ${\mathcal{C}}_{\mathcal{F}}(x)$.

The proof of Theorem~\ref{thm:not-generic-LY} is thus completed.
\end{proof}

Therefore by Theorems~\ref{thm3.2} and~\ref{thm:not-generic-LY},
it follows that every orbit $Gx$ admits at least one F-chaotic pair if its minimal $\mathcal{F}$-center
of attraction ${\mathcal{C}}_{\mathcal{F}}(x)$ is not a minimal subset of $(G,X)$.

Next we shall show that another kind of chaotic dynamics appears inside a minimal
center of attraction which is S-generic but not minimal for $(G,X)$.
Recall that 
\begin{itemize}
\item $(G,X)$ is called {\it Auslander-Yorke chaotic}
if it is point transitive and sensitive to initial conditions (see~\cite{Auslander-Yorke} for $\mathbb{Z}_+$-systems and \cite{Dai-2} for $\mathbb{R}_+$-systems);
$(G,X)$ is {\it topologically ergodic}
if for any nonempty open sets $U$ and $V$ of $X$,
$N(U,V)$ is syndetic in $G$.
\end{itemize}
The following is a generalization of \cite[Lemma~3.5]{Dai-2}.

\begin{lem}\label{lem3.6}
If ${\mathcal{C}}_{\mathcal{F}}(x)= X$,
then $(G,X)$ is topologically ergodic.
\end{lem}

\begin{proof}
Given any two nonempty open sets $U$, $V$ of $X$.
By Lemma~\ref{lem1.2}, the orbit $Gx$ is dense in $X$.
From Proposition~\ref{prop:point-tran-equi-trans},
one can take some $g\in G$  so that $U':= U \cap g^{-1}V\neq\emptyset$.
Whence
$N(U, V)\supset gN(U',U')$
and so it is sufficient to show that $N(U',U')$ is syndetic.
For this,
let $P$ with $P\not=G$ be any thick subset of $G$.
Now, we need to show that $N(U',U')\cap P\neq\emptyset$.
By Lemma~\ref{lem1.2} and Proposition~\ref{prop:sup-subnet},
it follows
that there exists a subnet $\{F_{n_k}\}_{k\in E}$ of $\mathcal{F}$ such that
\[\lim_{k\in E}\frac{|N(x,U')\cap F_{n_k}|}{|F_{n_k}|}>0.\]
Then by a standard argument of ergodic theory,
it follows that there exists a $G$-invariant Borel probability measure
$\mu$ on $X$ such that $\mu(U') > 0$.

%{\it Case 1.} $e\in P$. In this case, it is trivial that $N(U',U')\cap P\neq\varnothing$.
%{\it Case 2}. $e\notin P$.

Choose an element $p_1\in P$ with $p_1\not=e$.
Since $P$ is thick, there exists $p_2\in G$ with $p_2\not=e$ such that $p_2$, $p_1p_2\in P$.
Again from the thickness of $P$,
there exists $p_3\in G$ with $p_3\not=e$ such that
\[p_3,p_1p_3,p_2p_3, p_1p_2p_3\in P.\]
Inductively, we obtain a sequence $\{p_n\}_1^\infty$ with $p_n\not=e$ and
with
\[p_1,p_2, p_1p_2, p_3,p_1p_3,p_2p_3, p_1p_2p_3,\ldots \in P.\]
For each $n\in\mathbb{N}$,
take $g_n=p_1\cdots p_n$.

If $\{g_n\}_1^\infty$ has only finite many distinct values in $G$,
then there exist $m,n$ with $m<n$ such that $g_n=g_m$.
Thus  $g_nU'=g_mU'$ and $p_{m+1}\cdots p_n\in N(U',U')$.
Therefore $N(U',U')\cap P\neq\emptyset$ as $p_{m+1}\cdots p_n\in P$.

Now assume $\{g_n\}_1^\infty$ has infinite many distinct values in $G$.
In this case,
there exists a subsequence $\{g_{n_k}\}_{k=1}^\infty$ with $g_{n_i}\neq g_{n_j}$ for every $i\neq j$.
Since $\mu(U^\prime)>0$ and $\mu$ is a probability invariant measure,
it follows that $\mu(g_{n_{k'}}U^\prime\cap g_{n_{k''}}U^\prime)>0$ for some $k'>k''$.
This implies that $p_{n_{k''}+1}\cdots p_{n_{k'}}\in N(U',U')$.
Therefore we again have $N(U',U')\cap P\neq \varnothing$.

This proves Lemma~\ref{lem3.6}.
\end{proof}

The following result asserts that the Auslander-Yorke chaotic dynamics occurs inside a minimal center of attraction
if it is S-generic but not minimal for $(G,X)$.

\begin{thm}\label{thm:generic-AY}
If ${\mathcal{C}}_{\mathcal{F}}(x)$ is an S-generic
non-minimal subset of $(G,X)$, then $(G,{\mathcal{C}}_{\mathcal{F}}(x))$ is
point transitive and sensitive in the following sense:
one can find an $\epsilon>0$ such that
for any $z\in X$ there exist a dense subset $U_\epsilon(z)$ of $X$ such that for each $y\in U_\epsilon(z)$
there exists a sequence $\{g_n\}_1^\infty$ in $G$ such that $\lim_{n\to+\infty}d(g_n z, g_n y)\geq \epsilon$.
Specially, $(G,{\mathcal{C}}_{\mathcal{F}}(x))$ is Auslander-Yorke chaotic.
\end{thm}

\begin{proof}
Suppose that ${\mathcal{C}}_{\mathcal{F}}(x)$ is an S-generic
non-minimal subset of $(G,X)$.
Without loss of generality, we may assume that ${\mathcal{C}}_{\mathcal{F}}(x)=X$.
By Lemma~\ref{lem1.2},
it follows that the system $(G,X)$ is point transitive and $x$ is a transitive point.
Hence  we only
need to prove that the sensitivity of $(G,X)$.

Since $(G,X)$ is not minimal, by Lemma~\ref{lem1.2},
 one can find a minimal subset, say $K$, of $(G,X)$.
Fix some point $w\in X$ with $\alpha:=d(w, K)>0 $.
Let $r>0$ with $r={\alpha}/{3}$ and $U$ be any nonempty open subset of
$X$.
As $x$ is a transitive, there exists $t\in G$ such that $\hat{x}:=tx\in U$.
It is clear that $\hat{x}$ is also a transitive point.
Since $K$ is $G$-invariant,
it follows that the set $N(\hat{x}, B_r(K))$ is thick by Proposition~\ref{prop:invar-thick}.
Since $N(U, B_r(w))$ is syndetic in $G$ by Lemma~\ref{lem3.6},
we can choose an element
$g_1\in N(\hat{x}, B_r(K))\cap N(U,B_r(w))$;
and choose a nonempty open set $U_1$ in $X$ with
\[c\ell_X{U}_1\subset U\quad \text{and}\quad
c\ell_X{U}_1\subset U \cap g_1^{-1}(B_r(w)).\]
Then we can choose some element $g_2$ such that $g_1\neq g_2$ and
$g_2\in N(\hat{x}, B_r(K))\cap N(U_1,B_r(w))$;
and choose a nonempty open set $U_2$ in $X$ with $c\ell_X{U}_2\subset U_1$ and such that
$c\ell_X{U}_2\subset U_1 \cap g_2^{-1}(B_r(w))$.
Repeating this construction without end, we can then find a sequence of open nonempty sets $\{U_n\}_1^\infty$
in $X$ and a sequence $\{g_n\}_1^\infty$ of $G$ with
\[U\supset c\ell_X{U}_1\supset U_1\supset c\ell_X{U}_2\supset U_2\supset\cdots,\quad c\ell_X{U}_{n+1}\subset U_n \cap g_{n+1}^{-1}(B_r(w)).\]
and $g_i\neq g_j$ for each $i\neq j$.
Then $\cap_{n=1}^\infty U_n\neq\emptyset$ and for any $\hat{y}\in\cap_{n=1}^\infty U_n$, there are
\[\hat{y}\in U \quad \text{and}\quad d(g_n \hat{x}, g_n \hat{y})\geq r\  \forall n\in \mathbb{N}.\]
Let $\epsilon={r}/{2}$.
Then for any $z\in X$ and any nonempty open set $U$,
one can find some point $y\in U$ and a sequence $\{g_n\}_1^\infty$ in $G$ such that
\[\lim_{n\to+\infty}d(g_n z, g_n y)\geq \epsilon.\]
Finally for any $z\in X$,
let
\[U_\epsilon(z)=\left\{y\in X\,|\,{\lim}_{n\to+\infty}d(g_n z, g_n y)\geq \epsilon \text{~for some sequence~} \{g_n\}_1^\infty \text{~in~} G\right\}.\]
Since the open subset $U$ is arbitrary,
it follows that $N_\epsilon(z)$ is a dense subset in $X$.

The proof of Theorem~\ref{thm:generic-AY} is thus completed.
\end{proof}
%%%%%%%%%%%%%%%%%%%%%%%%%%%%%%%%%%%%%%%%%%%
%%%%%%%%%%%%%%%%%%%%%%%%%%%%%%%%%%%%%%%%%%%
\section{Li-Yorke chaotic pairs and sensitivity (II)}\label{sec4}%%%

Inspired by~\cite{X,zh-ye,WCF-15}, in this section,
we will show that if ${\mathcal{C}}_{\mathcal{F}}(x)$ is S-generic and non-minimal,
then ${\mathcal{C}}_{\mathcal{F}}(x)$ exhibits more complicated sensitivity than the non-S-generic case in Theorem~\ref{thm3.4}.

Let $(G,X)$ be a $G$-system and $\mathcal{F}=\{F_n\}_{n\in D}$ a F{\o}lner net in $G$.
For any tuple $(x_1,\ldots,x_n)$ in $X^n$,
we define a subset of $X$, write $L(x_1,\ldots,x_n)$, by
$x\in L(x_1,\ldots,x_n)$ if and only if for any $U_i\in\mathscr{U}_{x_i}$ and $U\in\mathscr{U}_x$,
there exist $x_i'\in U$ and $g\in G$ such that $(gx_1',\ldots,gx_n')\in U_1\times\cdots\times U_n$.

For our convenience, let us restate Theorem~\ref{thm1.6} as follows:

\begin{rthm1.6}
Let $G$ be commutative.
If ${\mathcal{C}}_{\mathcal{F}}(x)$ is S-generic and non-minimal,
and almost periodic points of $(G,X)$ are dense in ${\mathcal{C}}_{\mathcal{F}}(x)$,
then $(G,X)$ has $\aleph_0$-sensitivity near ${\mathcal{C}}_{\mathcal{F}}(x)$ in the following sense:
\begin{itemize}
\item One can find  an infinite countable subset $K$ of ${\mathcal{C}}_{\mathcal{F}}(x)$
such that for
any $k$ distinct points $x_1,\ldots,x_k\in K$ with $k\geq 2$,
there holds ${\mathcal{C}}_{\mathcal{F}}(x)\subseteq L(x_1,\ldots,x_k)$.
\end{itemize}
\end{rthm1.6}

\begin{proof}
Since ${\mathcal{C}}_{\mathcal{F}}(x)$ is S-generic,
there exists a point $y\in {\mathcal{C}}_{\mathcal{F}}(x)$ with ${\mathcal{C}}_{\mathcal{F}}(y)={\mathcal{C}}_{\mathcal{F}}(x)$
and thus $y$ is a transitive point in the subsystem $(G,{\mathcal{C}}_{\mathcal{F}}(x))$. We will divide our discussion into $5$ claims.

\begin{claim4}\label{cl4.1}
For any $z\in {\mathcal{C}}_{\mathcal{F}}(x)$
and $U\in\mathscr{U}_z$, $N(y,U)$ is piecewise syndetic in $G$.
\end{claim4}

\begin{proof}
By \cite[Lemma~3.7]{Dai-3}, it follows that one can find a syndetic set $S$ in $G$ such that for any finite set $B\subset S$, there is some $g_b^{}\in G$ with $Bg_b^{}\subset N(y,U)$. In addition, there is a finite set $F\subset G$ with $F^{-1}S=G$. Now for any finite set $A\subset G$, we can write
\begin{gather*}
A=A_1\cup\dotsm\cup A_n\quad \textrm{with }A_i\subset f_i^{-1}S, f_i\in F,\ i=1,\dotsc,n.
\end{gather*}
Then for $B=f_1A_1\cup\dotsm\cup f_nA_n\subset S$, we have $Bg_a^{}\subset N(y,U)$ for some $g_a^{}\in G$. Then for all $i=1,\dotsc,n$,
\begin{gather*}
f_iA_ig_a^{}\subset N(y,U)
\end{gather*}
and so 
\begin{gather*}
A_ig_a^{}\subset f_i^{-1}N(y,U)\subseteq F^{-1}N(y,U).
\end{gather*}
Thus $Ag_a^{}\subset F^{-1}N(y,U)$. This proves Claim~\ref{cl4.1}.
\end{proof}

\begin{claim4}\label{cl4.2}%%%
There are infinite many distinct minimal subsets in ${\mathcal{C}}_{\mathcal{F}}(x)$.
\end{claim4}

\begin{proof}
This follows from the density of almost periodic points of $(G,X)$ in ${\mathcal{C}}_{\mathcal{F}}(x)$.
\end{proof}

Now let $\{M_k\}_{k=1}^\infty$ be a sequence of minimal subsets of ${\mathcal{C}}_{\mathcal{F}}(x)$ with $M_i\neq M_j$
for every $i\neq j$.

\begin{claim4}\label{cl4.3}
For any $\delta>0$, $z\in {\mathcal{C}}_{\mathcal{F}}(x)$, and any $U\in\mathscr{U}_z$, both
$N(y,B_\delta(M_k)), N(U,B_\delta(M_k))$ are thick in $G$ for each $k\ge1$.
\end{claim4}

\begin{proof}
Since $M_k$ is an invariant closed subset,
Proposition~\ref{prop:invar-thick} follows that $N(y,B_\delta(M_k))$ and then $N(U,B_\delta(M_k))$ both are thick in $G$.
\end{proof}

\begin{claim4}\label{cl4.4}
For any $k\geq 2$, $\delta>0$, $z\in {\mathcal{C}}_{\mathcal{F}}(x)$, and any $U\in\mathscr{U}_z$, we have
$N(U,B_\delta(M_1))\cap\cdots\cap N(U, B_\delta(M_k))\neq\emptyset$.
\end{claim4}

\begin{proof}
Let $k\geq 2, \delta>0, z\in {\mathcal{C}}_{\mathcal{F}}(x)$, and $U\in\mathscr{U}_z$ be arbitrarily given.
Since $y$ is a transitive point in ${\mathcal{C}}_{\mathcal{F}}(x)$, it follows that
$N(B_\delta(M_i),U)=N(y,U)N(y,B_\delta(M_i))^{-1}$ for $1\le i\le k$.
By Claim~\ref{cl4.1}, it follows that there exists a finite subset $F$ of $G$ such that
for every finite subset $L$ of $G$ there exists $g_L\in G$ with
$Lg_L\subset F^{-1}N(y,U)$.
Then by Claim~\ref{cl4.3}, it follows that for each $1\leq i\leq k$,
we can choose $g_i\in G$ such that
$Fg_i\subset N(y,B_\delta(M_i))$.
Take $L^\ast=\{g_i\}_{i=1}^k$.
Then there exists  $g_{L^\ast}\in G$ with
$L^\ast g_{L^\ast}\subset F^{-1}N(y,U)$.
Thus for each $i=1,\dotsc,k$, there exists $f_i\in F$ such that $f_ig_ig_{L^*}\in N(y,U)$
which implies that
\begin{align*}
N(y,U)N(y,B_\delta(M_i))^{-1} &\supset \bigcup_{g\in F}N(y,U)(gg_i)^{-1}\supset\left\{f_ig_ig_{L^*}(gg_i)^{-1}\,|\,g\in F\right\}\ni f_ig_i g_{L^\ast}(f_ig_i)^{-1}
=g_{L^\ast}.
\end{align*}
Thus $N(B_\delta(M_1),U)\cap\cdots\cap N(B_\delta(M_k),U)\neq\emptyset$
and further
$N(U,B_\delta(M_1))\cap\cdots\cap N(U,B_\delta(M_k))\neq\emptyset$.
This proves Claim~\ref{cl4.4}.
\end{proof}

\begin{claim4}\label{cl4.5}%%%
For each $k\geq 2$, there exist
 $x_1\in M_1,\dotsc,x_k\in M_k$ such that ${\mathcal{C}}_{\mathcal{F}}(x)\subset L(x_1,\ldots,x_k)$.
\end{claim4}

\begin{proof}
Fix an arbitrary $k\geq 2$.
For each $n\in \mathbb{N}$, by Claim~\ref{cl4.4},
there exist
$y_{1,n},\ldots,y_{k,n}\in B_{{1}/{n}}(y)$
and $g_n\in G$
such that
\[g_ny_{1,n}\in B_{{1}/{n}}(M_1), \ldots, g_ny_{k,n}\in B_{{1}/{n}}(M_k).\]
Take
\[x_{1,n}=g_ny_{1,n},\ldots,x_{k,n}=g_ny_{k,n}.\]
Without loss of generality, we may assume that
\[x_1=\lim_{n\rightarrow\infty}x_{1,n},\ldots,x_k=\lim_{n\rightarrow\infty}x_{k,n}.\]
Then $x_1\in M_1$, \ldots, $x_k\in M_k$ and $y\in L(x_1,\ldots,x_k)$.

Let $z$ be an arbitrary point in ${\mathcal{C}}_{\mathcal{F}}(x)$.
Let $U_1,\ldots,U_k$ and $U$ be open neighborhoods of $x_1,\ldots,x_k$ and $z$ respectively.
Since $y$ is a transitive point,
there exists $l\in G$ such that $l^{-1}U$ is an open neighborhood of $y$.
Because $y\in L(x_1,\ldots,x_k)$,
there exist $y_1,\ldots,y_k\in l^{-1}U$ and $r\in G$ such that
$ry_1\in U_1,\ldots,ry_k\in U_k$ which implies that
\[(rl^{-1})ly_1\in U_1,\ldots,(rl^{-1})ly_k\in U_k.\]
Take $z_1=ly_1$,\ldots,$z_k=ly_k$  and $l^\ast=rl^{-1}$.
Then $z_1,\ldots,z_k\in U$ and
\[(l^\ast z_1,\ldots,l^\ast z_k)\in U_1\times \cdots\times U_k.\]
Thus $z\in L(x_1,\ldots,x_k)$.
Since $z$ is arbitrary, it follows that
\[{\mathcal{C}}_{\mathcal{F}}(x)\subset L(x_1,\ldots,x_k).\]
This ends the proof of Claim~\ref{cl4.5}.
\end{proof}

Finally for each $k\geq 2$, by Claim~\ref{cl4.5},
there exist points $x_{k,1},\ldots,x_{k,k}\in X$ so that
$x_{k,1}\in M_1,\ldots,x_{k,k}\in M_k$
and such that
${\mathcal{C}}_{\mathcal{F}}(x)\subset L(x_{k,1},\ldots,x_{k,k})$.
Thus
$\{x_{2,1},x_{3,1},\ldots\}\subset M_1$.
Assume $\lim_{n\to+\infty}x_{n,1}=x_1\in M_1$ by consider a subsequence of $\{n_i\}$ if necessary.
Then
$\{x_{2,2},x_{3,2},\ldots\}\subset M_2$.
Assume that $\lim_{n\to+\infty}x_{n,2}=x_2\in M_2$.
Continuing this construction, we obtain an infinite countable set $K=\{x_1,x_2,\ldots\}$.
Again from Claim~\ref{cl4.5}, it follows that for any
$x_1,\ldots,x_k\in K$,
there holds
${\mathcal{C}}_{\mathcal{F}}(x)\subset L(x_1,\ldots,x_k)$ and this ends the proof of Theorem~\ref{thm1.6}.
\end{proof}

\begin{proof}[Proof of Theorem~\ref{thm1.5}]
First of all, we note that all of Claims~\ref{cl4.1}, \ref{cl4.2} and \ref{cl4.3} hold without the commutativity of $G$. Now
$N(U, B_\delta(M_i))\cap N(U, B_\delta(M_j))\not=\emptyset$ for $i\not=j$
by Claim~\ref{cl4.3} and Lemma~\ref{lem3.6}. Then the proof of Claim~\ref{cl4.5} implies Theorem~\ref{thm1.5}.
\end{proof}
%%%%%%%%%%%%%%%%%%%%%%%%%
%%%%%%%%%%%%%%%%%%%%%%%%%
\section{Three examples}\label{sec5}%%%
In this section, we will firstly construct a simple example to exhibit that
if the minimal center of attraction of a $G$-system is not S-generic then it
may not admit the complex dynamics described in Theorems~\ref{thm:generic-AY}, \ref{thm1.5} and \ref{thm1.6}. Secondly, we shall present an example to show that the proof approach of Theorem~\ref{thm1.5} may not work for $\aleph_0$-sensitivity without the commutativity. Moreover, we will construct an example to show that relative to different F{\o}lner sequences in $\mathbb{Z}$, the minimal centers of attraction of a same point may be different.

\subsection{Example}\label{sec5.1}%%%
Let $\varSigma_{2}=\bigl\{x=(x_{n})_{-\infty}^{\infty}\,|\, x_{n}\in\{0, 1\}\bigl\}$.
Consider a metric on $\varSigma_{2}$ given by
\[d(x,y)=2^{-n},\quad \text{where } n=\min\{i\geq 0\,|\, x_i\neq y_i \text{ or } x_{-i}\neq y_{-i}\}.\]
With respect to this metric, $\varSigma_{2}$ is homeomorphic to the Cantor set.
The standard {\it shift} $\sigma$ on $\varSigma_2$ is defined by
$\sigma\colon\varSigma_2\rightarrow \varSigma_2$,
$(x_{n})_{n=-\infty}^{\infty}\mapsto (x_{n+1})_{n=-\infty}^{\infty}$.
Then $\sigma$ is a homeomorphism from $\varSigma_2$ to itself and $(\sigma,\varSigma_2)$
is a symbolic dynamical system.

A {\it word} is a finite sequence of elements of $\{0,1\}$.
Fix an arbitrary $n\in\mathbb{N}$.
For any given words $w=w_1w_2\cdots w_n$, $u=u_1u_2\cdots u_{2n+1}$ and any $m\in\mathbb{Z}$,
denote
\[[w]_m^n=\{x\in\varSigma_{2}\,|\, x_mx_{m+1}\cdots x_{m+n-1}=w \}\]
and $|w|$ by the {\it cylinder} of $w$ at position $m$ and the {\it length} of $w$ respectively.
For simplicity, we  use $[u]_0$ instead of $[u]_{-n}^{2n+1}$.
The collection of cylinders forms a basis of the topology of $\varSigma_{2}$.
Denote $\overline{w}:=w_nw_{n-1}\cdots w_1$ for any word $w=w_1w_2\dotsm w_n$.

Take
\begin{align*}
&A_1=01, A_2=A_1\underbrace{0000\cdots 0}_{2|A_1|-times},
A_3=A_2\underbrace{1111\cdots 1}_{3|A_2|-times}, \ldots, A_{2n}=A_{2n-1}\underbrace{0000000\cdots 0}_{(2n)|A_{2n-1}|-times},\\
&A_{2n+1}=A_{2n}\underbrace{11111111\cdots 1}_{(2n+1)|A_{2n}|-times}, \ldots.
\end{align*}
Then
 \[[\overline{A_1}A_1]_0\supset [\overline{A_2}A_2]_0\supset\cdots\supset [\overline{A_n}A_n]_0\supset\cdots\]
 and
\[\lim_{n\rightarrow\infty} \text{diam}([\overline{A_n}A_n]_0)=0.\]
By the compactness of $\varSigma_{2}$,
the set $\bigcap_{n=0}^\infty [\overline{A_n}A_n]_0$  contains exactly one element, denote it by $x$.
For each $n\in \mathbb{N}$,
take $F_n=\{-n,\ldots, n\}$.

Let $0^\infty=(0)_{-\infty}^{+\infty}$ and $1^\infty=(1)_{-\infty}^{+\infty}$. It is not hard to verify that the set $\{0^\infty,1^\infty\}$ is closed and invariant in $(\sigma,\varSigma_2)$.
We claim that  ${\mathcal{C}}_{\mathcal{F}}(x)=\{0^\infty,1^\infty\}$.
Indeed, fix an arbitrary $k\geq 2$.
For any largely $M>|A_1|$,
there exists $r\in \mathbb{N}$ such that $|A_{r+1}|\geq M> |A_r|$
and then if $M-|A_r|\geq 2k+1$, we have
\begin{align*}
\frac{|N(x, [0^\infty]_{-k}^{2k+1}\cup [1^\infty]_{-k}^{2k+1})\cap F_{M}|}{|F_{M}|}
&\geq \frac{2(M-|A_{r-1}|-(2k+1))}{2M+1}\\
&\geq 1-\frac{|A_{r-1}|+2(2k+1)}{2r|A_{r-1}|+1};
\end{align*}
if $M-|A_r|< 2k+1$, we have
\begin{align*}
\frac{|N(x, [0^\infty]_{-k}^{2k+1}\cup [1^\infty]_{-k}^{2k+1})\cap F_{M}|}{|F_{M}|}
\geq &\frac{2(r|A_{r-1}|-(2k+1))}{2(|A_r|+2k+1)}\\
\geq& \frac{r|A_{r-1}|-(2k+1)}{(r+1)|A_{r-1}|+2k+1}\\
=&1-\frac{|A_{r-1}|+2(2k+1)}{(r+1)|A_{r-1}|+2k+1}.
\end{align*}
Let $M\rightarrow+\infty$, then
\[\frac{|N(x, [0^\infty]_{-k}^{2k+1}\cup [1^\infty]_{-k}^{2k+1})\cap F_{M}|}{|F_{M}|}\rightarrow 1.\]
Since $0^\infty,1^\infty$ are fixed points in $(\sigma,{\mathcal{C}}_{\mathcal{F}}(x))$,
it follows that ${\mathcal{C}}_{\mathcal{F}}(x)$ is not S-generic in $(\sigma,\varSigma_2)$ and $(\sigma,{\mathcal{C}}_{\mathcal{F}}(x))$ is not point transitive.
And as ${\mathcal{C}}_{\mathcal{F}}(x)$ is a finite subset in $X$,
it follows that it does not admit the complex dynamics stated
in Theorems~\ref{thm:generic-AY}, \ref{thm1.5} and \ref{thm1.6}.

\subsection{Example}
In the proof of Theorem~\ref{thm1.5}, we have utilized a simple but crucial fact: If $A, B$ are two syndetic and thick subsets of $\mathbb{Z}$, then $A\cap B\not=\emptyset$ is piecewise syndetic. But this is never the case for more than two sets.

Indeed, let
\begin{align*}
A&=\bigcup_{n=1}^\infty A_n\cup 10\mathbb{N}&& \textrm{where }A_n=\{10^n,10^n+1,\ldots,10^n+10n-1\};\\
\intertext{and}
B&=\bigcup_{n=1}^\infty B_n\cup(10\mathbb{N}-1)&&\textrm{where }
B_n=\{10^n+10n,10^n+10n+2,\ldots,10^n+10n+n-1\}.
\end{align*}
Then
\[A\cap B= \big(\bigcup_{n=1}^\infty A_n\cap(10\mathbb{N}-1)\big)\cup \big(\bigcup_{n=1}^\infty B_n\cap10\mathbb{N}\big).\]
Note that
\begin{equation*}
\min\{a\,|\,a\in A_{n+1}\}-\max\{b\,|\,b\in B_{n}\}-n=9\cdot10^n+1-11n\to+\infty
\end{equation*}
as $n\rightarrow+\infty$.

In addition let
\[C=\bigcup_{n=1}^\infty C_n\cup (10\mathbb{N}-2)\ \textrm{where }
C_n=\{10^n+10n+n, 10^n+10n+n+1,\ldots,10^{n+1}-1\}.\]
Then $C$ is not only thick but also syndetic in $\mathbb{N}$ by the above construction.
However,
\[C_n\cap A_m=\emptyset\quad \text{and}\quad C_n\cap B_m=\emptyset\]
for all $m,n\in \mathbb{N}$.
Thus $A\cap B\cap C=\emptyset$.

Denote $-D=\{-d\,|\,d\in D\}$ for any subset $D$ of $\mathbb{Z}$.
Let
\[A^\ast=A\cup (-A),\quad B^\ast=B\cup (-B)\quad \text{and}\quad C^\ast=C\cup (-C).\]
Then the sets $A^\ast$, $B^\ast$ and $C^\ast$ are all thick and syndetic in $\mathbb{Z}$ with $A^*\cap B^*\cap C^*=\emptyset$.

Therefore, there exist three subsets $A,B,C$ of $\mathbb{Z}$, which are all syndetic and thick in $\mathbb{Z}$,
such that $A\cap B\cap C=\emptyset$.

\subsection{Example}\label{sec5.3}%%%
We now return to consider the canonical two-sided shift system $\sigma\colon\varSigma_2\rightarrow \varSigma_2$.
For any $n=0,1,2,\dotsc$, let
\begin{align*}
a_n&=n(n+1), &&b_n=n(n+2)+1;\\
R_n&=\{a_n,a_n+1,\ldots,a_n+n\}, &&R_n'=\{-a_n,-a_n-1,\ldots,-a_n-n\};\\
\intertext{and}G_n&=\{b_n,b_n+1,\ldots,b_n+n\}, && G_n'=\{-b_n,-b_n-1,\ldots,-b_n-n\}.
\end{align*}
Define two sequences $\mathcal{F}:=\{F_n\}_0^\infty$ and $\mathcal{H}:=\{H_n\}_0^\infty$ of finite subsets of $\mathbb{Z}$ by
$$F_{n}=\left\{\begin{array}{ll}
R_k, & \text{if } n=2k, \\
R_k', & \text{if }  n=2k+1;
\end{array}\right.$$
and
$$H_{n}=\left\{\begin{array}{ll}
G_k, & \text{if } n=2k, \\
G_k', & \text{if }  n=2k+1.
\end{array}\right.$$
Put $F=\bigcup_{n=0}^\infty F_n$ and $z=\mathbf{1}_F\in\varSigma_2$ where $\mathbf{1}_F(i)=1$ if $i\in F$ and
$\mathbf{1}_F(i)=0$ if $i\notin F$.
Then
\begin{enumerate}
\item $\mathcal{F}$ and $\mathcal{H}$ both are F{\o}lner sequences in $(\mathbb{Z},+)$;
\item ${\mathcal{C}}_{\mathcal{F}}(z)=\{1^{\infty}\}$ and ${\mathcal{C}}_{\mathcal{H}}(z)=\{0^{\infty}\}$.
\end{enumerate}

\begin{proof}
(1) Let $h\in \mathbb{Z}$.
Then
\[\lim_{n\rightarrow \infty}\frac{|(h+F_n)\vartriangle F_n|}{|F_n|}=\lim_{n\rightarrow \infty}\frac{2|h|}{|F_n|}=0\]
which implies that $\{F_n\}_0^\infty$ is a F{\o}lner sequence in $(\mathbb{Z},+)$.
Similarly, $\{H_n\}_0^\infty$ is also a F{\o}lner sequence in $(\mathbb{Z},+)$.

(2) It is easy to check that
$\{1^{\infty}\}$ is an $\sigma$-invariant closed subset of $\varSigma_{2}$.
Fix any $n\in \mathbb{N}$.
Then
\[B_{{1}/{2^n}}(\{1^{\infty}\})=\{x\in \varSigma_{2}\,|\,x_{-n}\cdots x_{-1}x_0 x_{1}\cdots x_n=\stackrel{n\textrm{-times}}{\overbrace{1\cdots 1}}1\stackrel{n\textrm{-times}}{\overbrace{1\cdots 1}}\}.\]
Thus
\[\frac{|N(z,B_{{1}/{2^n}}(\{1^{\infty}\}))\bigcap F_m|}{|F_m|}=\frac{|F_m|-2n-1}{|F_m|}\]
for every $m>4n$.
Hence
\[\lim_{m\rightarrow \infty}\frac{|N(z,B_{{1}/{2^n}}(\{1^{\infty}\}))\bigcap F_m|}{|F_m|}=1.\]
Therefore ${\mathcal{C}}_{\mathcal{F}}(z)=\{1^{\infty}\}$.

Similarly, we can obtain that ${\mathcal{C}}_{\mathcal{H}}(z)=\{0^{\infty}\}$. This thus completes the construction of Example~\ref{sec5.3}.
\end{proof}
%%%%%%%%%%%%%%%%%%%%%%%%%%%%%%%%%%%
%%%%%%%%%%%%%%%%%%%%%%%%%%%%%%%%%%%
\section*{Acknowledgements}
This work was supported partly by National Natural Science Foundation of China grants $\sharp$11431012 and $\sharp$11271183. The authors would like to thank the referee for her/his many helpful comments.

%%%%%%%%%%%%%%%%%%%%%%%%%%%

\bigskip

\begin{thebibliography}{MM}
\bibitem{Argabright}
   \newblock {L.~Argabright and C.~Wilde},
   \newblock {\it Semigroups satisfying a strong F{\o}lner condition}.
   \newblock {Proc. Amer. Math. Soc. \textbf{18} (1967), 587--591}.


\bibitem{Auslander-Yorke}
   \newblock {J.~Auslander and J.A.~Yorke},
      \newblock {\it Interval maps, factors of maps, and chaos}.
         \newblock {T\^{o}hoku Math. J. \textbf{32}(1980), 177--188.}

\bibitem{Banks-sensitive}
   \newblock {J.~Banks, J.~Brooks, G.~Cairns, G.~Davis and P.~Stacey},
      \newblock {\it On Devaney¡¯s definition of chaos}.
         \newblock {Amer. Math. Monthly  \textbf{99} (1992), 332--334.}

\bibitem{BHM}
   \newblock {V.~Bergelson, N.~Hindman and R.~McCutcheon},
   \newblock {\it Notions of size and combinatorial properties of quotient sets in semigroups}.
   \newblock {Topol. Proc. \textbf{23} (1998), 23-60}.

\bibitem{BM}
   \newblock {V.~Bergelson and R.~McCutcheon},
   \newblock {\it Recurrence for semigroup actions and a non-commutative Schur theorem}.
   \newblock {Topological Dynamics and Applications, Contemp. Math., Amer. Math. Soc., Providence, \textbf{215} (1998), 205-222}.

\bibitem{CD}
   \newblock {B.~Chen and X.~Dai},
   \newblock {\it On uniformly recurrent motions of topological semigroup actions}.
   \newblock {Discret. Contin. Dyn. Syst. \textbf{36} (2016), 2931--2944}.

\bibitem{Dai-1}
   \newblock {X.~Dai},
   \newblock {\it Chaotic dynamics of continuous-time topological semiflow on Polish spaces}.
   \newblock {J. Differential Equations \textbf{258} (2015), 2794--2805}.

\bibitem{Dai-2}
   \newblock {X.~Dai},
   \newblock {\it On chaotic minimal center of attraction of a Lagrange stable motion for topological semi flows}.
   \newblock {J. Differential Equations \textbf{260} (2016), 4393--4409}.

\bibitem{Dai-3}
   \newblock {X.~Dai},
   \newblock {\it Gr\"{u}nwald version of van der Waerden's theorem for semimodules}.
   \newblock {arXiv:1512.08695v2 [math.DS] 29 Dec 2015}.
   
\bibitem{DT}
   \newblock {X.~Dai and X.~Tang},
   \newblock {\it Devaney chaos, Li-Yorke chaos, and multi-dimensional Li-Yorke chaos for topological dynamics}.
   \newblock {J. Differential Equations \textbf{xxx} (2017), xx+33 http://dx.doi.org/10.1016/j.jde.2017.06.021}.

\bibitem{Devaney-89}
   \newblock {R.L.~Devaney},
   \newblock {\it An Introduction to Chaotic Dynamical Systems}, 2nd ed.
   \newblock {Addison-Wesley Studies in Nonlinearity, Addison-Wesley Publishing Company Advanced Book Program, Redwood City, CA, 1989}.

\bibitem {Ellis-semigroup}
   \newblock {D.~Ellis, R.~Ellis and M.~Nerurkar},
   \newblock {\it The topological dynamics of semigroup actions}.
   \newblock {Trans. Amer. Math. Soc.  \textbf{353} (2001), 1279--1320}.

\bibitem{F-81}
   \newblock {H.~Furstenberg},
   \newblock {\it Recurrence in Ergodic Theory and Combinatorial Number Theory}.
   \newblock {M.B.~Porter Lectures, Princeton University Press, Princeton, 1981}.

\bibitem{G93}
   \newblock {E.~Glasner and B.~Weiss},
      \newblock {\it Sensitive dependence on initial conditions}.
         \newblock {Nonlinearity \textbf{6} (1993), 1067--1075.}

\bibitem{Hilmy-36}
   \newblock {H.F.~Hilmy},
      \newblock {\it Sur les centres d¡¯attraction minimaux des syst`emes dynamiques}.
         \newblock {Compositio Math. \textbf{3} (1936), 227--238.}

\bibitem{Hindman}
   \newblock {N.~Hindman and D.~Strauss},
      \newblock \textit{Algebra in the Stone-\v{C}ech compactification: Theory and Applications}.
         \newblock de Gruyter, Berlin, 1998.

\bibitem{Hindman-06}
   \newblock {N.~Hindman and D.~Strauss},
      \newblock \textit{Density in arbitrary semigroups}.
         \newblock Semigroup Forum \textbf{73} (2006), 273--300.

\bibitem{HZ-12}
   \newblock {Y.~Huang and Z.-L.~Zhou},
      \newblock \textit{Two new recurrent levels for $C^0$-flows}.
         \newblock Acta Appl. Math. \textbf{118} (2012), 125--145.

\bibitem{KM}
   \newblock {E.~Kontorovich and M.~Megrelishvili},
   \newblock \textit{A note on sensitivity of semigroup actions}.
   \newblock {Semigroup Forum \textbf{76} (2008), 133-141}.

\bibitem{LY}
   \newblock {J.~Li and X.-D.~Ye},
   \newblock {\it Recent development of chaos theory in topological dynamics}.
   \newblock {Acta. Math. Sinica (Engl. Ser.) \textbf{32} (2016), 83--114}.


\bibitem{LY-75}
   \newblock {T.~Li and J.~Yorke},
      \newblock {\it Period $3$ implies chaos}.
         \newblock {Amer. Math. Monthly  \textbf{82} (1975), 985--992.}

\bibitem{Nemytskii-60}
   \newblock {V.V.~Nemytskii and V.V.~Stepanov},
      \newblock {\it Qualitative Theory of Differential Equations}.
         \newblock {Princeton University Press, Princeton, New Jersey 1960.}
\bibitem{Sigmund-77}
   \newblock {K.~Sigmund},
      \newblock {\it On minimal centers of attraction and generic points}.
         \newblock {J. Reine Angew. Math. \textbf{295} (1977), 72--79}.
\bibitem{WCF-15}
   \newblock {H.~Wang, Z.~Chen and H.~Fu},
      \newblock {\it M-systems and scattering systems of semigroup actions}.
         \newblock {Semigroup Forum  \textbf{91} (2015), 699--717.}


\bibitem{X}
   \newblock {J.~Xiong},
   \newblock {\it Chaos in topological transitive systems}.
   \newblock {Sci. China: Math. \textbf{48} (2005), 929--939}.

\bibitem{zh-ye}
   \newblock {X.-D.~Ye and R.~Zhang},
   \newblock {\it On sensitive sets in topological dynamics}.
   \newblock {Nonlinearity \textbf{21} (2008), 1601--1620}.
\end{thebibliography}
\end{document}